\documentclass[12 pt, reqno]{amsart}
\usepackage{amsmath}%
\usepackage{amsfonts}%
\usepackage{amssymb}%

\usepackage{graphicx}
\usepackage{tikz}
\usepackage{caption}
\usepackage{subcaption}
\usepackage{sidecap}
\usepackage{multirow}
\usepackage{empheq}

\newtheorem{theorem}{Theorem}[section]
\theoremstyle{plain}

\newtheorem{corollary}[theorem]{Corollary}

\newtheorem{definition}[theorem]{Definition}
\newtheorem{example}[theorem]{Example}

\newtheorem{lemma}[theorem]{Lemma}
\newtheorem{notation}[theorem]{Notation}
\newtheorem{observation}[theorem]{Observation}

\newtheorem{proposition}[theorem]{Proposition}
\newtheorem{remark}[theorem]{Remark}

\newcommand{\X}{\mathcal{X}}
\newcommand{\Y}{\mathcal{Y}}
\newcommand{\A}{\mathbb{A}}
\newcommand{\Q}{\mathbb{Q}}
\newcommand{\OO}{\mathcal{O}}
\newcommand{\C}{\mathbb{C}}
\newcommand{\R}{\mathbb{R}}
\newcommand{\Z}{\mathbb{Z}}
\newcommand{\PP}{\mathbb{P}}
\newcommand{\N}{\mathbb{N}}

\title[Anticanonical models of smoothings of CQS]{Anticanonical models of smoothings of cyclic quotient singularities}
\author{Ari\'e Stern}

\email{stern@math.umass.edu}

\begin{document}

\begin{abstract}
Given a surface cyclic quotient singularity $Q\in Y$, it is an open problem to determine all smoothings of $Y$ that admit an anticanonical model and to compute it. In \cite{HTU}, Hacking, Tevelev and Urz\'ua studied certain irreducible components of the versal deformation space of $Y$, and within these components, they found one parameter smoothings $\Y \to \A^1$ that admit an anticanonical model and proved that they have canonical singularities. Moreover, they compute explicitly the anticanonical models that have terminal singularities using Mori's division algorithm \cite{M02}.
We study one parameter smoothings in these components that admit an anticanonical model with canonical but non-terminal singularities with the goal of classifying them completely. We identify certain class of  ``diagonal" smoothings where the total space is a toric threefold and  we construct the anticanonical model explicitly using the toric MMP.
\end{abstract}

\maketitle

\section{Introduction}
Let $X$ be a smooth projective variety defined over $\C$, and let $K_X$ be the canonical divisor. Then the graded ring 
$$R(X,K_X)=\bigoplus_{m\ge 0} H^0(X,mK_X)$$
is called the \emph{canonical ring}. It is a birational invariant and Birkar-Cascini-Hacon-McKernan proved in \cite{BCHM} that it is a finitely generated graded algebra over $\C$, thus proving the existence of the \emph{canonical model}: Proj $R(X,K_X)$.
Similarly, one can consider the \emph{anticanonical ring} 
$$R(X,-K_X)=\bigoplus_{m\ge 0} H^0(X,-mK_X).$$
This ring is not a birational invariant and is not necessarily finitely generated, see \cite{S82} for a discussion in the case of ruled surfaces. When the anticanonical ring is finitely generated, then Proj $R(X,-K_X)$ is called the \emph{anticanonical model} of $X$.
Two varieties isomorphic in codimension one clearly have the same anticanonical model (if it exists).  Anticanonical models have been proven to be useful in the study of birational properties of Fano threefolds (see \cite{I80}), and a characterization of varieties of Fano type is given in \cite{CG} in terms of the singularities of their anticanonical models. 
\vspace{0.3cm}

A surface cyclic quotient singularity $Q\in Y$ is a germ at the origin of the quotient of $\C^2$ by a group action $(x,y)\mapsto (\mu x, \mu^q y)$ where $\mu$ is a primitive $m$-th root of unity, $1\leq q <m$ and $gcd(m,q)=1$. We denote this singularity by $\frac{1}{m}(1,q)$. Let $\tilde{Y} \to Y$ be the minimal resolution of $Y$, then on $\tilde{Y}$ we have a chain of exceptional curves $E_i$, $1\leq i \leq s$, such that $E_i^2=-b_i$ where the numbers $b_i$ appear in the Hirzebruch-Jung continued fraction 
$$\frac{m}{q}=[b_1,\ldots ,b_s ].$$
\vspace{0.2cm}

We are interested in anticanonical models of smoothings $\Y$ of $Y$ over a disc. The canonical model of $\Y$, which is given by a deformation of a P-resolution of $Y$ (see below), is very useful in the study of deformations of $Y$ \cite{KSB88} and moduli spaces of surfaces of general type \cite{HTU}. We do not assume $\Y$ is $\Q$-Gorenstein, so $K_{\Y}$ is not $\Q$-Cartier in general. The variety $\Y$ is normal and not projective, so we have to define what we mean by an anticanonical model of $\Y$ in this case. Note that $K_{\Y}$ is not Cartier so $\OO (-nK_{\Y})$ is a divisorial sheaf, not a line bundle. We define the anticanonical model of $\Y$ as 
$$\text{Proj}_{\Y} \bigoplus_{n \geq 0} \OO (-nK_{\Y})$$
under the condition that this is a sheaf of finitely-generated algebras. 
\vspace{0.3cm}

A normal surface singularity is called a $T$-singularity if it is a quotient singularity and it admits a $\Q$-Gorenstein smoothing. An explicit description of these singularities can be found on \cite{KSB88}. $T$-singularities include Wahl singularities, i.e., cyclic quotient singularities of the form
$$\frac{1}{m^2}(1,ma-1)$$
where $1\leq a <m$ and $gcd(m,a)=1$. By \cite[Theorem 3.9]{KSB88}, there is a correspondence between irreducible components in the versal deformation space of a cyclic quotient singularity $Y$ and P-resolutions of $Y$, i.e., partial resolutions $f^+ \colon X^+ \to Y$, such that $X^+$ has only T-singularities and $K_{X^+}$ is relatively ample. Any deformation $\Y$ of $Y$ within the corresponding component is obtained by blowing down a $\Q$-Gorenstein deformation $\X^+$ of $X^+$, which gives the canonical model of $\Y$.  
\vspace{0.3cm}

Extremal P-resolutions (introduced in \cite{HTU}) will be particularly important to us. An extremal P-resolution is a P-resolution $X^+ \to Y$ with the additional properties that the exceptional set is a curve\\ $C^+ \simeq \PP^1$, and $X^+$ has at most two Wahl singularities $$(P \in X^+) \simeq \A^2/ \frac{1}{m^2}(1,ma-1)$$ along $C^+$. By \cite{HTU}, a cyclic quotient singularity $Q\in Y$ admits at most two extremal P-resolutions. One-parameter $\Q$-Gorenstein smoothings of extremal P-resolutions have important numerical invariants called axial multiplicities (see Definition \ref{axial}) which determine the deformation locally around the singular points of $\X^+$.

\vspace{0.3cm}

Hacking, Tevelev and Urz\'ua proved in \cite[Corollary 3.23]{HTU}, that a smoothing $\Y \to \A^1$ of $Y$ admits an anticanonical model if it is obtained by blowing down a smoothing $\X^+ \to \A^1$ of an extremal P-resolution $X^+$. Let $-c$ be the selfintersection of the proper transform of $C^+$ in the minimal resolution of $X^+$ and let $$\delta=cm_1'm_2'-m_1'a_2'-m_2'a_1'.$$ Let $\alpha_1, \alpha_2$ be the axial multiplicities of the singularities of $\X^+$. If $\alpha_1^2 -\delta \alpha_1 \alpha_2 + \alpha_2^ 2 >0$ then the anticanonical model has terminal singularities and it can be computed explicitly using Mori's division algorithm \cite{M02} and \cite{HTU}.
Let $$\mathcal{R}=\{(\alpha_1,\alpha_2) | \alpha_1^2 -\delta \alpha_1 \alpha_2 + \alpha_2^ 2 \leq 0 \}$$ We will call this set the \textit{Canonical Region} (see Figure 2).
If $(\alpha_1,\alpha_2)\in \mathcal{R}$, then the anticanonical model has non-terminal singularities \cite{HTU}, and no explicit construction or description is known. We work out the diagonal case, i.e. when $\alpha_1=\alpha_2$, and
we prove by a non-trivial change of coordinates that $\X^+$ and $\Y$ are toric threefolds. The map $\X^+ \to \Y$ is toric, but the induced deformation of $X^+$ is not toric in the sense of \cite{A95}. Then using tools from toric geometry, we prove

\begin{theorem}
\label{diagonal}
Let $\X^+$ be a $\Q$-Gorenstein smoothing of an extremal $P$-resolution with axial multiplicities $\alpha_1=\alpha_2$, and singularities  $$\frac{1}{m_1'^2}(1,m_1'a_1'-1), \qquad \frac{1}{m_2'^2}(1,m_2'a_2'-1)$$ Let $\langle w_1,w_2,w_3 \rangle$ be a basis of $N=\Z^3$. Then there exist vectors $$w_4=m_1'w_1+w_2+cw_3 \text{  and  } w_5=-m_2'w_1+w_2+dw_3$$ for some $c,d\in \Z$ such that $\Y, \X^+$ and the anticanonical model $\X^-$ are analytically isomorphic to toric varieties given by the fans in Figure 1. 
In this way we get $\X^-=W_1 \cup W_2$ where $W_1=\frac{1}{\delta}(-\rho -1,\rho,1)$ and $W_2=\frac{1}{F}(\lambda,1,-1)$ for some $\rho, \lambda \in \Z$ (given explicitly in the proof) and $F=m_1'+m_2'$.
\end{theorem}

\begin{figure}[h]
\centering
\includegraphics[scale=0.15]{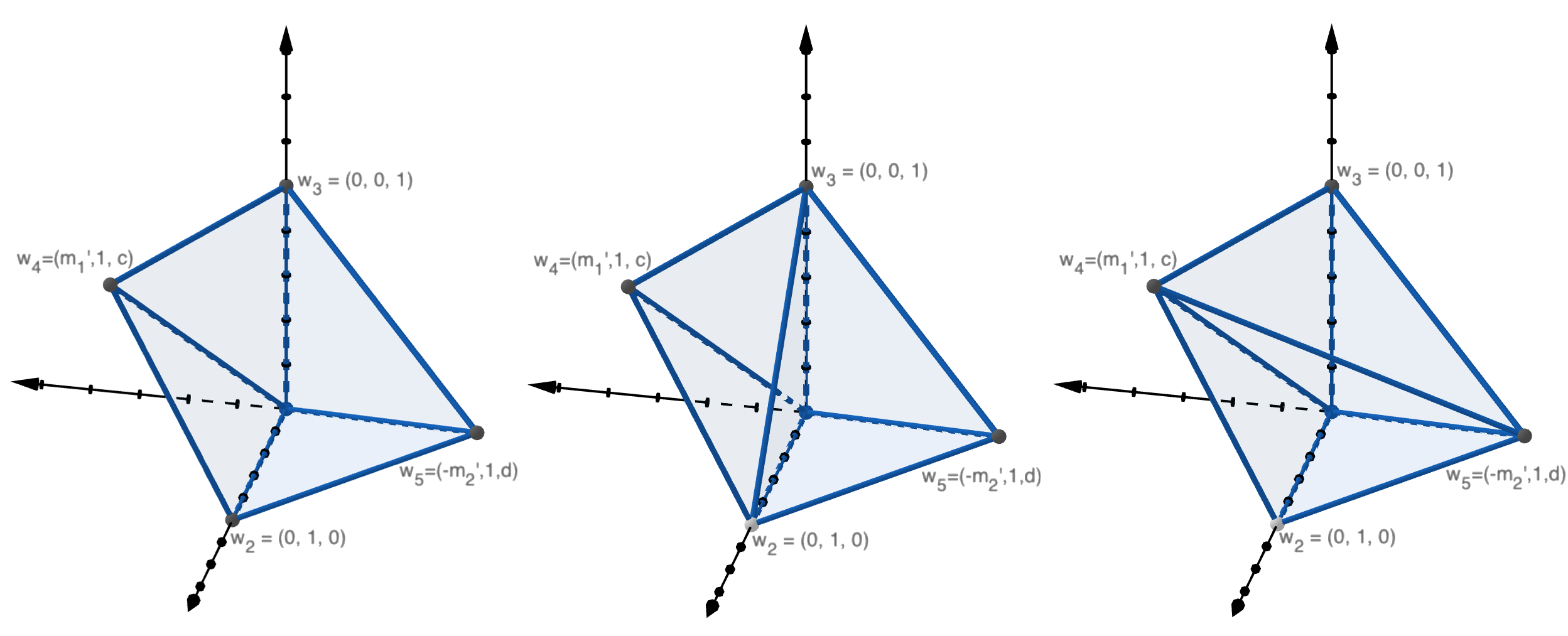}
\caption{Fans of $\Y$, $\X^+$ and $\X^-$}
\end{figure}

By the classification of Ishida and Iwashita (\cite{IsIw}), we have that the singularities of $\X^-$ are canonical, confirming \cite[Remark 3.25]{HTU}. Note that
$W_1$ is non-terminal and $W_2$ is terminal if and only if $gcd(F,\lambda)=1$.

Then we analyze the special fiber of $\X^-$. We say that a Hirzebruch-Jung continued fraction $[a_1,\ldots , a_r]$ is conjugate to $[b_1,\ldots , b_s]$ if $[a_1,\ldots , a_r]$ is the continued fraction of $\frac{m}{m-q}$.
An important type of surfaces that will be used repeatedly on this paper, are orbifold normal crossings. This surfaces are defined by an equation 
$$(YZ=0)\subset \A^3/\frac{1}{m}(1,a,-a)$$
and they have two braches with two conjugate cyclic quotient singularities 
$$\frac{1}{m}(1,a) \quad \text{and} \quad \frac{1}{m}(1,-a).$$
An orbifold normal crossing singularity admits a terminal smoothing defined as follows:
$$(YZ=tf(X^m))\subset \A^3/\frac{1}{m}(1,a,-a) \times A^1_t$$
with $f(0)\neq 0$.
\vspace{0.3cm}

We analyze cases $\delta \geq 3$ and $\delta=2$ separately. We prove

\begin{theorem}
\label{specialfiber}
If $\delta \geq 3$, then the special fiber $X^-$ is a non normal surface, singular along a curve $C^- \simeq \PP^1$. A transversal slice of $\X^-$ through a general point of $C^-$ is a surface with an $A_{\delta -1}$ singularity.
Let $X^{\nu}$ be the normalization of $X^-$ and let $\bar{X}$ be the minimal resolution of $X^{\nu}$. Then $\bar{X}$ has the following configuration of smooth rational curves:
\begin{equation}
[a_r, \ldots , a_1] - \circ - [c_1, \ldots , c_l] - \circ - [b_1, \ldots , b_s]
\label{eq 1}
\end{equation}
where the chains $[a_1, \ldots, a_r]$ and $[b_1, \ldots, b_s]$ correspond to conjugate cyclic quotient singularities and $\circ$ denotes a $(-1)$-curve. $X^{\nu}$ is obtained from $\bar{X}$ by contracting the $[a],[b],[c]$ chains of rational curves. Finally to obtain $C^-$ and $X^-$ we glue the two $(-1)$-curves creating a orbifold normal crossing point (the image of the [a] and [b] chains). Locally around the non-terminal point (the image of the [c] chain), $X^-$ is given by the equation 
\begin{equation}
(XYZ=Y^{\delta}+ Z^{\delta})\subset A^3/\frac{1}{\delta}(-\rho -1, \rho, 1).
\end{equation}
\end{theorem}

\begin{remark} In \cite{MP}, Mori and Prokhorov classified terminal threefold extremal contractions of type $(IA)$ and $(IA^{\vee})$. These threefolds are germs of an irreducible curve $C$ which has negative intersection with the canonical divisor, and they study a general element $H\in |\OO_C(-K)|$. In the case when $H$ is not normal, they proved that the minimal resolution of the normalization of $H$ has only two possible configurations of exceptional divisors. One of these possibilities consists of a chain of curves as in (\ref{eq 1}). The difference is that the threefolds considered in \cite{MP} are terminal, which is equivalent to the condition that 
$$\sum (c_i-2)\leq 2.$$
Our threefolds are non-terminal so this condition is never true. It would be interesting to classify  canonical but non-terminal extremal neighborhoods having these configuration of exceptional curves. 
\vspace{0.2cm}

\noindent
The non-terminal point (2) is a quotient of a degenerated cusp singularity of type $T^3_{\delta-2,\delta-2}$  according to the notation in \cite{Tzi}.
\end{remark}

Next we consider the case $\delta =2$ when the canonical region $\mathcal{R}$ coincides with the diagonal line.

\begin{theorem}
\label{delta2}
There is a two-to-one correspondence between\\

\begin{empheq}[left=\empheqlbrace, right=\empheqrbrace]{align}
 &\Q \text{-Gorenstein smoothings } X^+\subset \X^+ \text{ with } \delta=2 \text{ and axial } \nonumber \\
 & \text{ multiplicities } \alpha_1=\alpha_2 \text{, except the case of Proposition \ref{rnc4}}\nonumber
\end{empheq}

and the set

\noindent
$$\bigg \{ \{(p,f)\in \Z^2 \hspace{0.1cm} | \hspace{0.1cm} 1\leq p\leq \frac{f}{2} , gcd(p,f)=1\} \bigg \}.$$

\noindent
The special fiber of the antiflip $X^-\subset \X^-$ is a non normal surface, singular along $C^- \simeq \PP^1$. A transversal slice of $\X^-$ through a general point of $C^-$ is a surface with an $A_1$ singularity. 
Let $X^{\nu}$ be the normalization of $X^-$ and $\bar{X}$ is its minimal resolution, then $\bar{X}$ has the following configuration of curves:
$$[a_r, \ldots , a_1] - C - [b_1, \ldots , b_s]$$
where the chains $[a_1, \ldots, a_r]$ and $[b_1, \ldots, b_s]$ correspond to  the conjugate cyclic quotient singularities 
$$\frac{1}{f}(p,1), \text{ and } \frac{1}{f}(p,-1)$$
and $C$ is the proper transform of $C^-$ in $\bar{X}$ and $C^2=-5$.  $X^-$ is obtained from $\bar{X}$ by first contracting the $[a]$ and $[b]$ chains of rational curves and then folding the curve $C$ onto itself, producing an orbifold normal crossing point, two pinch points, and the singularity with local equation $$(XYZ=Y^2+Z^2)\subset A^3/\frac{1}{2}(-\rho -1, \rho, 1).$$
\end{theorem}

\begin{corollary}
If $Q\in Y$ is a cyclic quotient singularity that admits an extremal P-resolution with $\delta=2$, then $Q\in Y$ has exactly two extremal P-resolutions, unless $Q\in Y$ is the cone over the rational normal curve of degree four, which has only one extremal P-resolution. 
\end{corollary}

\begin{proposition}
\label{rnc4}
Let $Q\in Y$ be the cone over the rational normal curve of degree $4$, and let $X^+$ be its minimal resolution. We have $m_i'=a_i'=1$, $i=1,2$, so $\delta=2$. Then $X^-$ has two pinch points along $C^-$, a transversal slice of $\X^-$ through a general point of $C^-$ is a surface with an $A_1$ singularity. The normalization of $X^-$ is a smooth surface $X^{\nu} \simeq X^+$.
\end{proposition}

See Corollary \ref{cor rnc} for the cone over the rational normal curve of degree greater than four.

\begin{remark}
When $\delta=2$ the canonical region $R$ is equal to the diagonal. So Theorem \ref{delta2} and Proposition \ref{rnc4} together with \cite{HTU} give a complete description of the anticanonical models of $\X^+$ in this case. 
\end{remark}

\begin{example}
In Table \ref{table} we show how the correspondence of Theorem \ref{delta2} works for $2 \leq f \leq 7$. For each pair $(f,p)$ we give the two extremal P-resolutions $X^+$ and the surface $X^{\nu}$. 
As it will be proven in Lemma \ref{cases}, the proper transform of $C^+$ in the minimal resolution of $X^+$ has self-intersection $(-1)$ or $(-2)$  and $(C^-)^2=(-5)$. The singular points on $C^+$ and $X^{\nu}$ are represented by their corresponding continued fractions.
\end{example}

\begin{table}[]
\begin{tabular}{|c|c|c|c|}
\hline
f                  & p                  & $X^+$             & $X^{\nu}$   \\ \hline
\multirow{2}{*}{2} & \multirow{2}{*}{1} & (-2)-{[}5,2{]}                   & \multirow{2}{*}{{[}2{]}-(-5)-{[}2{]}}       \\ \cline{3-3}
                   &                    & {[}2,5{]}-(-2)                   &                                             \\ \hline
\multirow{2}{*}{3} & \multirow{2}{*}{1} & {[}3,5,2{]}-(-2)                 & \multirow{2}{*}{{[}3{]}-(-5)-{[}2,2{]}}     \\ \cline{3-3}
                   &                    & {[}4{]}-(-1)-{[}6,2,2{]}         &                                                                                          \\ \hline
\multirow{2}{*}{4} & \multirow{2}{*}{1} & {[}4,5,2,2{]}-(-2)               & \multirow{2}{*}{{[}4{]}-(-5)-{[}2,2,2{]}}   \\ \cline{3-3}
                   &                    & {[}5,2{]}-(-1)-{[}7,2,2,2{]}     &                                              \\ \hline
\multirow{2}{*}{5} & \multirow{2}{*}{1} & {[}5,5,2,2,2{]}-(-2)             & \multirow{2}{*}{{[}5{]}-(-5)-{[}2,2,2,2{]}} \\ \cline{3-3}
                   &                    & {[}6,2,2{]}-(-1)-{[}8,2,2,2,2{]} &                                             \\ \hline
\multirow{2}{*}{5} & \multirow{2}{*}{2} & {[}3,2,6,2{]}-(-1)-{[}5,2{]}     & \multirow{2}{*}{{[}3,2{]}-(-5)-{[}3,2{]}}   \\ \cline{3-3}
                   &                    & {[}4{]}-(-1)-{[}3,5,3,2{]}       &                                              \\ \hline
\multirow{2}{*}{6} & \multirow{2}{*}{1} & {[}6,5,2,2,2,2{]}-(-2)             & \multirow{2}{*}{{[}6{]}-(-5)-{[}2,2,2,2,2{]}} \\ \cline{3-3}
                   &                    & {[}7,2,2,2{]}-(-1)-{[}9,2,2,2,2,2{]} &                                             \\ \hline
\multirow{2}{*}{7} & \multirow{2}{*}{1} & {[}7,5,2,2,2,2,2{]}-(-2)                 & \multirow{2}{*}{{[}7{]}-(-5)-{[}2,2,2,2,2,2{]}}     \\ \cline{3-3}
                   &                    & {[}8,2,2,2,2{]}-(-1)-{[}10,2,2,2,2,2,2{]}         &         
\\ \hline
\multirow{2}{*}{7} & \multirow{2}{*}{2} & {[}4,2,6,2,2{]}-(-1)-{[}6,2,2{]}                 & \multirow{2}{*}{{[}4,2{]}-(-5)-{[}3,2,2{]}}     \\ \cline{3-3}
                   &                    & {[}5,2{]}-(-1)-{[}4,5,3,2,2{]}         &       
\\ \hline
\multirow{2}{*}{7} & \multirow{2}{*}{3} & {[}3,2,2,7,2{]}-(-1)-{[}3,5,2{]}                 & \multirow{2}{*}{{[}3,2,2{]}-(-5)-{[}4,2{]}}     \\ \cline{3-3}
                   &                    & {[}4{]}-(-1)-{[}3,2,5,4,2{]}         &   
\\ \hline
\end{tabular}
\caption{}
\label{table}
\end{table}

\noindent
{\bf Acknowledgements.} I would like to thank my advisor, Jenia Tevelev, for his guidance during my thesis and this work. I would also like to thank Giancarlo Urz\'ua and Paul Hacking for helpful discussions. This project has been partially supported by the NSF grant DMS-1701704 (PI Jenia Tevelev).

\section{Construction of the flip and the diagonal case}

\subsection{Overview of Mori's algorithm \cite{M02,HTU}}

Let $X^+$ be an extremal $P$-resolution of a cyclic quotient singularity $Q\in Y$, and  let $\X^+$ be a $\Q$-Gorenstein smoothing of $X^+$.

\begin{definition}
\label{axial}
By Corollary 3.23 in \cite{HTU}, there is a divisor $D\in |-K_{\X^+}|$ such that the restriction $D|_{X^+}$ is equal to a chain of smooth rational curves $L_1-C^+-L_2 \subset X^+$. Let $P_1,P_2$ be points where $L_1$ and $L_2$ intersect $C^+$ respectively. The points $P_1, P_2$ are either smooth or a Wahl singularity.

\begin{itemize}
\item[(i)] If $P=\frac{1}{m^2}(1,ma-1)$ is a singular point of $X^+$, there is an analytic isomorphism (over $\C$) 
$$(P \in X^+) \simeq (\xi \eta =\zeta ^m) \subset A^3_{\xi, \eta, \zeta}/ \frac{1}{m}(1,-1,a)$$
and then for the deformation $(P\in \X^+) \to \A^1_t$  we get an analytic isomorphism
$$(P\in \X^+) \simeq (0\in \xi \eta = \zeta^m + t^{\alpha})\subset A^3_{\xi, \eta, \zeta}/ \frac{1}{m}(1,-1,a) \times \A^1_t$$
for some $\alpha\in \N$ called the axial multiplicity of $P\in \X^+$.

\item[(ii)] If $P$ is a smooth point of $X^+$ then the local deformation $(P_i\in D \subset \X^+) \to (0\in \A^1_t)$ of $(P_i\in C^+ \subset X^+ )$ is of the form 
$$(0\in (\xi \eta = t^{\alpha_i} h_i(t)) )\subset \A^2_{\xi , \eta} \times \A^1_t$$
for some $\alpha_i \in \N$ and convergent power series $h_i(t)$ with $h_i(0)\neq 0$. The number $\alpha_i$ is called the axial multiplicity of $P_i\in \X^+$.
\end{itemize}

\end{definition}

Let $\frac{1}{m_1'^2}(1,m_1'a_1'-1)$, $\frac{1}{m_2'^2}(1,m_2'a_2'-1)$ be the singularities of $X^+$ with Hirzebruch-Jung continued fractions $\frac{m_1'^2}{m_1'a_1'-1}=[e_1, \ldots , e_{r_1}]$ and $\frac{m_2'^2}{m_2'a_2'-1}=[f_1, \ldots , f_{r_2}]$. Then the singularity $Q\in Y$ is given by $$\frac{\Delta}{\Omega}=[f_{r_2},\ldots , f_1,c,e_1, \ldots , e_{r_1}]$$
where $-c$ is the self-intersection of the proper transform of $C^+$ in the minimal resolution of $X^+$. Define $$\delta=cm_1'm_2'-m_1'a_2'-m_2'a_1'$$ and define $$m_2=m_1',\qquad a_2=m_1'-a_1' \text{ if $m_1'\neq a_1'$, or } a_2=1 \text{ otherwise}$$ and $$m_1=\delta m_1' + m_2', \qquad a_1=\frac{\delta+m_1m_2-a_2m_1}{m_2}$$. 

The numbers $(m_1,a_1,m_2,a_2)$ represent an  ``initial" extremal neighbourhood $\X^{-}$ of type $k2A$ where the special fiber has Wahl singularities $\frac{1}{m
_1^2}(1,m_1a_1-1)$, $\frac{1}{m_2^2}(1,m_2a_2-1)$ and such that the flip of $\X^{-}$ is $\X^+$. In fact, as shown in \cite{HTU}, there is a toric surface $M$ (of locally finite type if $\delta>1$), corresponding to a fan $\Sigma$ with cones $\{0\}$, $\R_{\geq 0} v_i$, $\langle v_i, v_{i+1} \rangle _{\R_{\geq 0}}$ for some primitive vectors $v_i \in \Z^2$ defined as follows: 
$$v_1=(1,0) \text{, } v_2=(\delta,1),$$
$$v_{i+1}+v_{i-1}=\delta v_i$$

There is a toric birational morphism $p \colon M \to \A^2_{u_1'u_2'}$, flat irreducible families of surfaces $$\mathbb{X}^- \to M, \qquad  \mathbb{Y} \to M, \qquad   \mathbb{X}^+\to M$$ and morphisms 
$$\pi \colon \mathbb{X}^- \to \mathbb{Y} \times M,  \text{ and  } \pi^+ \colon \mathbb{X}^+ \to \mathbb{Y}.$$

There exist a morphism $g \colon \A^1_t \to M$ such that $p(g(0))=0\in \A^2$ and such that the flip of $\X^- \to \Y$ is the pullback of $\mathbb{X}^- \to \mathbb{Y}^+$ under $p \circ g$.  
Each $v_i$ has a label $(m_i,a_i)$. If we take a smoothing $\Y$ of $Y$ with axial multiplicites $(\alpha_1, \alpha_2)$ corresponding to one of the $v_i$'s (or corresponding to a ray which is between two consecutive primitive vectors $v_i$, $v_{i+1}$),  then we get an extremal neighbourhood of type $k1A$ (of type k2A, respectively). The special fiber has a Wahl singularity $\frac{1}{m_i^2}(1,m_ia_i-1)$ (respectively, has two Wahl singularities $\frac{1}{m_i^2}(1,m_ia_i-1),\frac{1}{m_{i+1}^2}(1,m_{i+1}a_{i+1}-1)$). All these extremal neighbourhoods have a flip with central fiber equal to $X^+$.

\begin{figure}[h]
\centering
\includegraphics[scale=0.3]{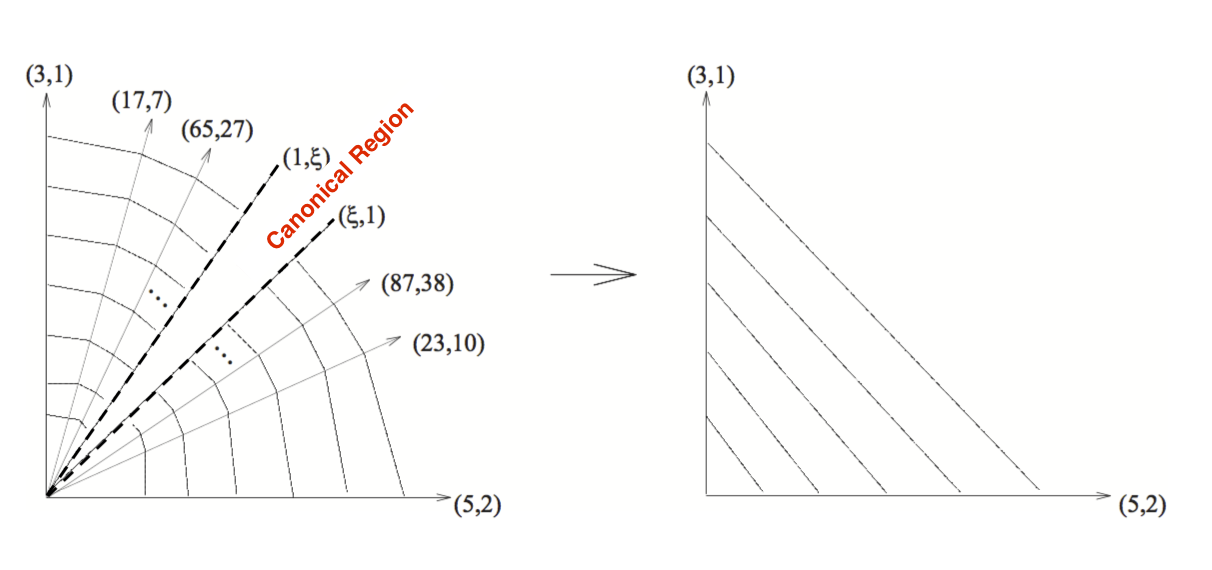}
\caption{Example of a map of fans $p \colon M \to \A^2$}
\end{figure}

\subsection{Construction of the flip}

Next we describe an explicit model for $\X^+$ following \cite{M02} and \cite{HTU}.
\begin{definition}
\label{seq}
 Define a sequence $$d(1)=m_1, \quad d(2)=m_2$$ and $$d(i+1)=\delta d(i)-d(i-1) \text{ for } i\ge 2.$$
By \cite[Lemma 3.3]{HTU}, there exists $k\ge 3$ such that $d(k-1)>0$ and $d(k)\leq 0$. Consider the sequence $c \colon \Z \to \Z$ defined as $$c(1)=a_1, \quad c(2)=m_2-a_2$$ and $$c(i-1)+c(i+1)=\delta c(i) \text{ for } 2\leq i \leq k-1.$$ 
We also define $c(k+1)=-c(k-1)$ and $c(k+2)=-c(k)$.
\end{definition}
\vspace{0.3cm}

Define 
$$W'=(x_1'y_1'=z^{m_1'}x_2'^{\delta}+u_1', x_2'y_2'=z^{m_2'}x_1'^{\delta}+u_2') \subset \A^5_{x_1',x_2',y_1',y_2',z}\times \A^2_{u_1',u_2'}$$
and 
$$\Gamma ' = \{\gamma'=(\gamma_1',\gamma_2')\quad  | \quad \gamma_1'^{m_1'}=\gamma_2'^{m_2'} \}\subset \mathbb{G}_m^2.$$
Define an action of $\Gamma'$ on $W'$ by 
$$\gamma \cdot (x_1',x_2',y_1',y_2',z,u_1',u_2') \mapsto (\gamma_1' x_1', \gamma_2' x_2', {\gamma_1'} ^{-1}y_1', {\gamma_2'}^{-1}, {\gamma_1'}^{c(k-1)}{\gamma_2'}^{c(k)}z, u_1',u_2').$$
Define 
$$W'^0 = W' \setminus (x_1'=x_2'=0)$$
and $$\mathbb{X}^+= (W'^0)/\Gamma'.$$
Write $U_1'=(x_2' \neq 0) \subset \mathbb{X}^+$ and $U_2'=(x_1'\neq 0)\subset \mathbb{X}^+$.\\ 
Then $\mathbb{X}^+=U_1' \cup U_2'$, 
$$U_i'= (\xi_i ' \eta_i '= {\zeta_i '}^{m_i'}+u_i' ) \subset \A^3_{\xi_i, \eta_i', \zeta_i'}/\frac{1}{m_i'}(1,-1,a_i') \times \A^2_{u_1',u_2'}$$
for each $i=1,2$, and glueing is given by 
$$U_1'\supset (\xi_1'\neq 0)=(\xi_2'\neq 0) \subset U_2',$$
$${\xi_1'}^{m_1'}={\xi_2'}^{-m_2'}, \qquad {\xi_1'}^{-c(k-1)}\zeta_1'= {\xi_2'}^{-c(k)}\zeta_2'.$$

\subsection{Proof of Theorem \ref{diagonal}}
Now we consider the diagonal case, corresponding to smoothings $\X^+$ where the axial multiplicities $\alpha_1$ and $\alpha_2$ are equal, which translates into $u_1=u_2=u$ in the equations of $\mathbb{X}^+$. The equations for $\X^+$ appear complicated, but we show that in the diagonal case $\X^+$ is a toric variety using a change of variables.

\begin{lemma}
$\X^+$ is a toric threefold.  
\end{lemma}

\begin{proof}
We will show that $\X^+$ is isomorphic to the variety $Z=V_1 \cup V_2$ where $$V_i=(\xi_i' \tilde{\eta_i} =q_i )\subset \A^3_{\xi_i' , \tilde{\eta_i} , \zeta_i}/\frac{1}{m_i'}(1,-1,a_i')\times \A^1_{q_i}$$ with gluing given by $$\xi_1^{m_1'}=\xi_2^{-m_2'}, \quad \xi_1' \tilde{\eta_1}=\xi_2' \tilde{\eta_2}$$
Consider the map $f_1 \colon \A^3_{\xi_1',\eta_1',\zeta_1'} \times \A^1_u \to \A^3_{\xi_1',\tilde{\eta}_1, \zeta_1'} \times \A^1_{q_1}$ given by $$(\xi_1',\eta_1',\zeta_1',u) \mapsto (\xi_1', \eta_1'+\xi_1'^{\delta-1}\zeta_1'^{m_2'},\zeta_1',  \zeta_1'^{m_1'}+\xi_1'^{\delta} \zeta_1'^{m_2'}+u).$$
\vspace{0.1cm}

\noindent
The group $\Z_{m_1'}$ acts on $\A^3_{\xi_1',\eta_1',\zeta_1'}$ with weights $1,-1,a_1'$ and on $\A^3_{\xi_1',\tilde{\eta}_1, \zeta_1'} $ with the same weights. Let $\gamma$ be a primitive $m_1'$-th root of unity, then 
\vspace{0.2cm}

\begin{equation*}
\begin{split}
& f_1(\gamma \cdot (\xi_1' , \eta_1' ,\zeta_1' , u))\\ 
& =(\gamma \xi_1' , \gamma^{-1}\eta_1' + \gamma^{\delta -1+a_1'm_2'}\xi_1'^{\delta-1} \zeta_1'^{m_2'}, \gamma^{a_1'}\zeta_1', \zeta_1'+ \gamma^{\delta + a_1'm_2'} \xi_1'^{\delta} \zeta_1'^{m_2'} + u)\\
& = (\gamma \xi_1' , \gamma^{-1}\eta_1' + \gamma^{-1}\xi_1'^{\delta-1} \zeta_1'^{m_2'}, \gamma^{a_1'}\zeta_1', \zeta_1'+ \xi_1'^{\delta} \zeta_1'^{m_2'} + u)\\
& = \gamma \cdot f_1(\xi_1', \eta_1', \zeta_1', u)
\end{split}
\end{equation*}

\noindent
where we have used that $\delta= cm_1'm_2' -m_1'a_2' - m_2'a_1'$ and then $\delta+m_2'a_1'$ is a multiple of $m_1'$. Therefore $f_1$ is equivariant and it descends to a map $f_1 \colon U_1' \to V_1$. 

\noindent
Similarly, consider the map  $f_2 \colon \A^3_{\xi_2',\eta_2',\zeta_2'} \times \A^1_u \to \A^3_{\xi_2',\tilde{\eta}_2, \zeta_2'} \times \A^1_{q_2}$ given by $$(\xi_2',\eta_2',\zeta_2',u) \mapsto (\xi_2', \eta_2'+\xi_2'^{\delta-1}\zeta_2'^{m_1'},\zeta_2',  \zeta_2'^{m_2'}+\xi_2'^{\delta} \zeta_2'^{m_1'}+u).$$

\vspace{0.1cm}

\noindent 
The group $\Z_{m_2'}$ acts on $\A^3_{\xi_2',\eta_2',\zeta_2'}$ with weights $1,-1,a_2'$ and on $\A^3_{\xi_2',\tilde{\eta}_2, \zeta_2'} $ with the same weights. A similar argument as above shows that $f_2$ is equivariant with respect to this action and therefore it descends to a map $f_2 \colon U_2' \to V_2$. 

Note that $f_1$ and $f_2$ are isomorphisms since we can write their inverses by expressing $\eta_i'$ and $u$ in terms of $\xi_i' , \tilde{\eta}_i, \zeta_i', q_i$, $i=1,2$. Therefore $U_i' \simeq V_i$. Now we show that the maps $f_1$, $f_2$ can be glued.   
\vspace{0.2cm}

Let $(\xi_1',\eta_1',\zeta_1',u)\subset (\xi_1' \neq 0) \subset U_1'$, then

\begin{eqnarray*}
\xi_1' \eta_1' &= &  \zeta_1'^{m_1'}+u \\  
\xi_1' (\tilde{\eta_1}-\xi_1'^{-1}\zeta_2'^{m_2'}) &= & \zeta_1'^{m_1'}+u\\
\xi_1' \tilde{\eta_1} &= & \zeta_1'^{m_1'} + \zeta_2'^{m_2'} + u\\
\xi_1' \tilde{\eta_1} &= & q_1.
\end{eqnarray*}

Similarly, if $(\xi_2',\eta_2',\zeta_2',u)\subset (\xi_2'\neq 0) \subset U_2'$, then 

\begin{eqnarray*}
\xi_2' \eta_2' &= &  \zeta_2'^{m_2'}+u \\  
\xi_2' (\tilde{\eta_2}-\xi_2'^{-1}\zeta_1'^{m_1'}) &= & \zeta_2'^{m_2'}+u\\
\xi_2' \tilde{\eta_2} &= & \zeta_1'^{m_1'} + \zeta_2'^{m_2'} + u\\
\xi_2' \tilde{\eta_2} &= & q_2.
\end{eqnarray*}

\noindent
Since $ U_1' \supset (\xi_1' \neq 0) =  (\xi_2'\neq 0) \subset U_2'$ we conclude that the gluing on $V_1\supset f_1( (\xi_1' \neq 0))=f_2((\xi_2'\neq 0))\subset V_2$ is given by 
$$\xi_1'^{m_1'}=\xi_2'^{-m_2'},  \quad \xi_1' \tilde{\eta_1}=\xi_2' \tilde{\eta_2}.$$

\noindent
Finally note that all the relations defining $Z$ are monomial relations and $Z$ is normal, so $Z$ is a toric variety. 
\end{proof}

\begin{lemma}
Let $N$ be the lattice generated by $\langle w_1,w_2,w_3 \rangle$ and let  $w_4=m_1'w_1+w_2-c(k-1)w_3 \text{ and } w_5=-m_2'w_1+w_2 -c(k)w_3$ (see Definition \ref{seq} for $c(k-1)$ and $c(k)$), then $\X^+$ is the toric variety corresponding to the fan with maximal cones $$\sigma_1=Cone(w_2, w_3, w_4) \text{ and } \sigma_2=Cone(w_2, w_3,w_5)$$   and $\Y$ is the toric variety corresponding to the  $Cone(w_2,w_3,w_4,w_5)$.
\end{lemma} 

\begin{proof}

Say $U_1=Spec(\C[S_{\sigma_1}])$ for $\sigma_1\subset N_{\R}$ where $\sigma_1^{\vee}=Cone(e_1^{\vee},e_2^{\vee},e_3^{\vee})$ and $U_2=Spec(\C[S_{\sigma_2}])$ for $\sigma_2 \in N_{\R}$ where $\sigma_2^{\vee}=Cone(f_1^{\vee},f_2^{\vee},f_3^{\vee})$.

\begin{notation}
\label{exp}
Given an element $r_1e_1^{\vee}+r_2e_2^{\vee}+r_3e_3^{\vee}\in S_{\sigma_1}$, we denote the corresponding element of its semigroup algebra $\C[S_{\sigma_1}]$ as ${\xi_1'}^{r_1}\tilde{\eta_1}^{r_2}{\zeta_1'}^{r_3}$. Similarly, to an element $s_1f_1^{\vee}+s_2f_2^{\vee}+s_3f_3^{\vee}\in S_{\sigma_2}$ we associate the element ${\xi_2'}^{s_1}\tilde{\eta_2}^{s_2}{\zeta_2'}^{s_3}\in \C[S_{\sigma_2}]$.
\end{notation}

The lattices of $U_1$ and $U_2$ are given by 
$$L_1^{\vee}=\{ \sum b_i e_i^{\vee} \colon b_1-b_2+a_1'b_3 \text{ is divisible by } m_1' , b_i \in \Z \}$$ and $$L_2^{\vee}=\{ \sum d_i f_i^{\vee} \colon d_1-d_2+a_2'd_3 \text{ is divisible by } m_2' , d_i \in \Z \}$$ respectively and from the glueing relations we get 
\begin{eqnarray*}
m_1'e_1^{\vee} &=& -m_2'f_1^{\vee}\\ 
-c(k-1)e_1^{\vee}+e_3^{\vee} &=& -c(k)f_1^{\vee}+f_3^{\vee}\\ 
e_1^{\vee}+e_2^{\vee} &=& f_1^{\vee}+f_2^{\vee}.
\end{eqnarray*}

\noindent
Let $w_1^{\vee}=m_1'e_1^{\vee}=-m_2'f_1^{\vee}$, $w_2^{\vee}=e_1^{\vee}+e_2^{\vee}=f_1^{\vee}+f_2^{\vee}$, and  $w_3^{\vee}=e_3^{\vee}-c(k-1)e_1^{\vee}=f_3^{\vee}-c(k)f_1^{\vee}$. 

We write the $e_i^{\vee}$'s and the $f_i^{\vee}$'s  in terms of the $w_i^{\vee}$'s and we use this to rewrite the cones as
$$\sigma_1^{\vee}=Cone \bigg( \frac{1}{m_1'} w_1^{\vee}, w_2^{\vee}-\frac{1}{m_1'}w_1^{\vee}, \frac{c(k-1)}{m_1'}w_1^{\vee}+w_3^{\vee} \bigg )$$ and 
$$\sigma_2^{\vee}=Cone \bigg (-\frac{1}{m_2'} w_1^{\vee},\frac{1}{m_2'}w_1^{\vee}+w_2^{\vee},w_3^{\vee}-\frac{c(k)}{m_2'}w_1^{\vee} \bigg ).$$

By Lemma \ref{lattice}, the lattices $L_1^{\vee}$ and $L_2^{\vee}$ are both equal to the lattice $N^{\vee}=\langle w_1^{\vee},w_2^{\vee},w_3^{\vee} \rangle$. 
Let $$w_4=m_1'w_1+w_2-c(k-1)w_3 \text{ and } w_5=-m_2'w_1+w_2 -c(k)w_3)$$

Then the duals of these cones are given by  
$$\sigma_1=Cone(w_2, w_3, w_4) \text{ and } \sigma_2=Cone(w_2, w_3,w_5)$$ 
which concludes the proof of the lemma.
\end{proof}

\begin{lemma}
\label{lattice}
The lattices $L_1^{\vee}$ and $L_2^{\vee}$ are both equal to the lattice $N^{\vee}=\langle w_1^{\vee},w_2^{\vee},w_3^{\vee} \rangle$.
\end{lemma}
\begin{proof}
It is clear that $N^{\vee} \subset L_1^{\vee}$ and $N^{\vee} \subset L_2^{\vee}$. 
Let $b_1e_1^{\vee}+b_2e_2^{\vee}+b_3e_3^{\vee} \in L_1^{\vee}$. 
Then 
$$b_1e_1^{\vee}+b_2e_2^{\vee}+b_3e_3^{\vee}= \bigg (\frac{b_1-b_2+c(k-1)b_3}{m_1'} \bigg )  w_1^{\vee}+b_2w_2^{\vee}+b_3w_3^{\vee}.$$
But this element is in $L_1^{\vee}$ so $b_1-b_2+a_1'b_3=m_1'l$ for some $l\in \Z$. In the construction of the flip in \cite{HTU}, $a_1'$ is defined as a number $0\leq a_1' \leq m_1'$ with $a_1 \equiv c(k-1) \text{ mod } m_1'$. Then the coefficient of $w_1^{\vee}$ in the above expression in an integer and $L_1^{\vee}=\{ \sum b_iw_i^{\vee} | b_i\in \Z \}=N^{\vee}$. 
\vspace{0.1cm}

\noindent
Similarly, if $d_1f_1^{\vee}+d_2f_2^{\vee}+d_3f_3^{\vee} \in L_2^{\vee}$ then writing 
we get 
$$d_1f_1^{\vee}+d_2f_2^{\vee}+d_3f_3^{\vee} = -\bigg ( \frac{d_1-d_2+c(k)d_3}{m_2'} \bigg ) w_1^{\vee} +d_2w_2^{\vee}+d_3w_3^{\vee}$$
But $d_1-d_2+a_2'd_3$ is divisible by $m_2'$ and by definition $a_2' \equiv c(k) \text{ mod } m_2'$. Therefore the coefficient of $w_1^{\vee}$ is an integer and $L_2^{\vee}=N^{\vee}$. 
\end{proof}

Note that the fan of $\X^+$ is obtained by subdividing the fan of $\Y$. There is only one more possible subdivision of the fan of $\Y$ given by $\sigma_3=Cone(w_2,w_4,w_5)$ and $\sigma_4=Cone(w_3,w_4,w_5)$. We will prove that these cones correspond to the two charts of the antiflip. 
\vspace{0.2cm}

We know that a toric variety given by a simplicial fan has only quotient singularities, and for each maximal cone $\sigma$, the affine open $U_{\sigma}$ is the quotient of $\C^n$ by the action of the finite abelian group $G=N/N'$ where $N'$ is the lattice obtained by the primitive generators of $\sigma$. Let $M, M'$ be the duals of $N$ and $N'$ respectively. Then we have an action of $G$ on $\C[M']$, determined by the canonical pairing

$$ M'/M \times N/N' \to \C^n$$
and given by
$$v(X^{u'})=exp(2\pi i \langle u',v\rangle )X^{u'}$$
for $v\in N$, $u' \in M'$. Then we consider the lattices 
\begin{eqnarray*}
L_3 & = & \langle  w_2,w_4,w_5\rangle \\
                 & = & \langle w_2,m_1'w_1-c(k-1)w_3, -m_2'w_1-c(k)w_3 \rangle  \subset N              
\end{eqnarray*}

and 

\begin{eqnarray*}
L_4 & = & \langle w_3,w_4, w_5 \rangle \\
                   & = & \langle w_3,m_1'w_1+w_2, -m_2'w_1+w_2  \rangle \subset N.
\end{eqnarray*}

\noindent
Let $G=N/L_3$ and $H=N/L_4$, then we have that $U_{\sigma_3}=\C^3/G$ and $U_{\sigma_4}=\C^3/H$.

Note that 

\[
\begin{bmatrix}
w_2\\
w_4\\
w_5
\end{bmatrix}=
\begin{bmatrix}
1 & 0 & 0\\
0 & m_1' & -c(k-1)\\
0 & -m_2' & -c(k)
\end{bmatrix} 
\begin{bmatrix}
w_2\\
w_1\\
w_3
\end{bmatrix}
\]
To determine $G$ we find the Smith normal form of this matrix. Recall that $\delta=-(m_1'c(k)+m_2'c(k-1))$. By definition, we have that $$c(k-1) \equiv a_1' \text{ (mod } m_1') \text{ and } gcd(a_1',m_1')=1.$$ We can find integers $r,s$ such that $rm_1'-sc(k-1)=1$. Let $\rho=rm_2'+sc(k)$, then we obtain the Smith normal form of the matrix by multiplying by the following invertible matrices

\[
\begin{bmatrix}
1 & 0 & 0\\
0 & 1 & 0\\
0 & \rho & 1
\end{bmatrix} 
\begin{bmatrix}
1 & 0 & 0\\
0 & m_1' & -c(k-1)\\
0 & -m_2' & -c(k)
\end{bmatrix} 
\begin{bmatrix}
1 & 0 & 0\\
0 & r  & c(k-1) \\
0 & s & m_1' 
\end{bmatrix} =
\begin{bmatrix}
1 & 0 & 0\\
0 & 1 & 0\\
0 & 0 & \delta
\end{bmatrix}
\]
so $G=\Z_{\delta}$ and is generated by the element $-sw_1+rw_3$. Now $L_3^{\vee} = \langle w_2^{\vee},w_6^{\vee},w_7^{\vee} \rangle$ where 
\begin{eqnarray*}
w_6^{\vee} &=& -\frac{c(k)}{\delta}w_1^{\vee}+\frac{m_2'}{\delta}w_3^{\vee} \\
w_7^{\vee} &=& \frac{c(k-1)}{\delta}w_1^{\vee}+\frac{m_1'}{\delta}w_3^{\vee} \\
\end{eqnarray*}
Now $\sigma_3^{\vee}=Cone(w_{10}^{\vee},w_6^{\vee}, w_7^{\vee})$ where 
$w_{10}^{\vee}= w_2^{\vee} -w_6^{\vee}-w_7^{\vee}$. Then $U_{\sigma_3}=\C^3/\Z_{\delta}$, the pairing is  given by 
$$\langle -sw_1+rw_3, w_{10}^{\vee} \rangle =\frac{-\rho-1}{\delta} \text{ , } \langle -sw_1+rw_3, w_6^{\vee} \rangle =\frac{\rho}{\delta} \text{ , } \langle -sw_1+rw_3, w_7^{\vee} \rangle =\frac{1}{\delta}$$
and then the weights are $-\rho-1, \rho, 1$.\\
For the other chart, note that 

\[
\begin{bmatrix}
w_3\\
w_4\\
w_5
\end{bmatrix}=
\begin{bmatrix}
1 & 0 & 0\\
0 & m_1' & 1\\
0 & -m_2' & 1
\end{bmatrix} 
\begin{bmatrix}
w_3\\
w_1\\
w_2
\end{bmatrix}
\]
To determine $H$ we find the Smith normal form of this matrix. Let $F=m_1'+m_2'$. Then we obtain the Smith normal form of the matrix by multiplying by the following invertible matrices

\[
\begin{bmatrix}
1 & 0 & 0\\
0 & 1 & 0\\
0 & 1 & -1
\end{bmatrix} 
\begin{bmatrix}
1 & 0 & 0\\
0 & m_1' & 1\\
0 & -m_2' & 1
\end{bmatrix} 
\begin{bmatrix}
1 & 0 & 0\\
0 & 0  & 1 \\
0 & 1 & -m_1' 
\end{bmatrix} =
\begin{bmatrix}
1 & 0 & 0\\
0 & 1 & 0\\
0 & 0 & F
\end{bmatrix}
\]
so $H=\Z_F$ and is generated by the element $w_1$.
Now $L_4^{\vee} = \langle w_3^{\vee}, w_8^{\vee}, w_9^{\vee} \rangle$ where 
\begin{eqnarray*}
w_8^{\vee} &=& \frac{1}{F} w_1^{\vee} + \frac{m_2'}{F}w_2^{\vee}  \\
w_9^{\vee} &=& -\frac{1}{F}w_1^{\vee}+\frac{m_1'}{F}w_2^{\vee}
\end{eqnarray*}
and $\sigma_4^{\vee}=Cone(w_{11}^{\vee},w_8^{\vee},w_9^{\vee})$ where
$$w_{11}^{\vee}= w_3^{\vee}+c(k-1) \bigg (\frac{1}{F}w_1^{\vee}+\frac{m_2'}{F}w_2^{\vee} \bigg ) +c(k) \bigg (-\frac{1}{F} w_1^{\vee}+\frac{m_1'}{F}w_2^{\vee} \bigg )$$
Then $U_{\sigma_4}=\C^3/\Z_F$, the pairing is  given by 
$$\langle w_1, w_{11}^{\vee} \rangle =\frac{\lambda}{F} \text{ , } \langle w_1, w_8^{\vee} \rangle =\frac{1}{F} \text{ , } \langle w_1, w_9^{\vee} \rangle =\frac{-1}{F}$$
and then the weights are $\lambda, 1, -1$, where $\lambda=c(k-1)-c(k)$. This concludes the proof of Theorem \ref{diagonal}.

\vspace{0.1cm}

\begin{corollary}
\label{cor rnc}
If $\X^+$ is the one parameter deformation of the minimal resolution of the rational normal curve of degree $n\geq 3$ with equal axial multiplicities, then the antiflip is $\X^-=W_1\cup W_2$ where $W_1=\frac{1}{n-2}(-2,1,1)$ and 
\begin{equation*}
  W_2 = \left\{
    \begin{array}{rl}
      \frac{1}{2}(0,1,1) & \text{if } n \text{ is even},\\
      \\
      \frac{1}{2}(1,1,1) & \text{if } n \text{ is odd}
    \end{array} \right.
\end{equation*}
\end{corollary}
\begin{proof}
Follows directly from the theorem since in this case we have  $m_1'=m_2'=a_1'=a_2'=1$, $\delta=n-2$, $c(k-1)=0$ and $c(k)=-\delta$.
\end{proof}

\section{Special fiber of the antiflip with $\delta\geq 3$}

\begin{lemma}
In the notation of Theorem \ref{diagonal}, the special fiber $X^-=S_1 \cup S_2$ where $$S_1=(X_1Y_1Z_1=Y_1^{\delta}+Z_1^{\delta})\subset \A^3/\frac{1}{\delta}(-\rho -1,\rho,1),$$ $$S_2=(Y_2Z_2=X_2^{m_1'}Z_2^{\delta} + X_2^{m_2'}Y_2^{\delta})\subset \A^3 /\frac{1}{F}(\lambda,1,-1)$$ and the gluing is given by $$X_1^{\delta}= X_2^{-F} \text{ ,  }Y_1^{\delta}=X_2^{m_2'}Y_2^{\delta} \text{ ,  } Z_1^{\delta}=X_2^{m_1'}Z_2^{\delta}.$$
\end{lemma}

\begin{proof}
To a linear combination $r_1w_{10}^{\vee}+r_2w_6^{\vee}+r_3w_7^{\vee} \in S_{\sigma_3}$ we associate the element $X_1^{r_1}Y_1^{r_2}Z_1^{r_3} \in \C[S_{\sigma_3}]$, and to a linear combination $r_1w_{11}^{\vee}+r_2w_8^{\vee}+r_3w_9^{\vee} \in S_{\sigma_4}$ we associate the element $X_2^{r_1}Y_2^{r_2}Z_2^{r_3} \in \C[S_{\sigma_4}]$ (see proof of Theorem \ref{diagonal} for the definition of the $w_i^{\vee}$) .

From the equations of $\X^+$ in Section 2.2, we have  $u=\xi_1' \tilde{\eta_1} - {\zeta_1'}^{m_1'} - {\zeta_2'}^{m_2'} \in \C[S_{\sigma_1}]$ which becomes $$u=X_1Y_1Z_1-Y_1^{\delta}-Z_1^{\delta} \in \C[S_{\sigma_3}]$$ or $$u=Y_2Z_2 - X_2^{m_1'}Z_2^{\delta} - X_2^{m_2'}Y_2^{\delta} \in \C[S_{\sigma_4}]$$
\end{proof}

\noindent
Let ${(X^-)}^{\nu},S_1^{\nu}$ and $S_2^{\nu}$ be the normalization of $X^-,S_1$ and $S_2$. Then $(X^-)^{\nu}=S_1^{\nu} \cup S_2^{\nu}$.
Let $$T_1=(X_1Y_1Z_1=Y_1^{\delta}+Z_1^{\delta})\subset \A^3$$
and $$T_2=(Y_2Z_2=X_2^{m_2'}Y_2^{\delta}+X_2^{m_1'}Z_1^{\delta})\subset \A^3.$$
Let $T_1^{\nu}$, $T_2^{\nu}$ be their normalizations, then $S_1^{\nu}$ is obtained by taking the quotient of $T_1^{\nu}$ by the $\Z_{\delta}$ action and $S_2^{\nu}$ is obtained by taking the quotient of $T_2^{\nu}$ by the $\Z_F$ action.

\begin{proof}[Proof of Theorem \ref{specialfiber}]
Let $$w=\frac{-Z_1^{\delta-1}+X_1Y_1}{Y_1}.$$
Note that $w$ is an integral element in $\C [X_1,Y_1,Z_1]$ since it satisfies the monic equation $w^2-wX_1+Y_1^{\delta-2}Z_1^{\delta-2}=0$. We will prove that $T_1^{\nu}$ is  the spectrum of  $$\C[X_1,Y_1,Z_1,w]/(wZ_1-Y_1^{\delta-1}, WY_1+Z_1^{\delta-1}-X_1Y_1, w^2-wX_1+Y_1^{\delta-2}Z_1^{\delta-2}).$$
\noindent
First we apply the following change of coordinates:  $X_1\mapsto \tilde{X}_1=w-X_1$ to get 
$$\C[\tilde{X}_1,Y_1,Z_1,w]/(wZ_1-Y_1^{\delta-1}, \tilde{X}_1Y_1+Z_1^{\delta-1}, \tilde{X}_1w+Y_1^{\delta-2}Z_1^{\delta-2}).$$
It is known that for a cyclic quotient singularity $\frac{1}{r}(1,a)$, where $a$ and $r$ are coprime, the invariant monimials are given by $u_0=p^r$, $u_1= p^{r-a}q$, $\ldots$ , $u_k=q^r$ and they satisfy the relations
$$u_{i-1}u_{i+1}=u_i^{a_i} \text{  for } i=1,\ldots, k$$
where the $a_i$ come from the continued fraction of $\frac{r}{r-a}=[a_1,\ldots , a_k]$. Consider the case when $r=\delta(\delta-2)$ and $a=(\delta-2)(\delta-1)-1$. Then we get invariants $u_0=p^{\delta(\delta-2)}$, $u_1=p^{\delta-1}q$, $u_2=pq^{\delta-1}, u_3=q^{\delta(\delta-2)}$ and they satisfy the relations 
$$u_0u_2=u_1^{\delta-1}, u_1u_3=u_2^{\delta-1}, u_0u_3=u_1^{\delta-2}u_2^{\delta-2}.$$
Notice that these equations are the same equations obtained above under the identification $$u_0 \mapsto \tilde{X}, u_1\mapsto Z_1, u_2 \mapsto Y_1 \text{ and } u_3 \mapsto w$$
Therefore the surface given by the Spec of 
$$\C[\tilde{X}_1,Y_1,Z_1,w]/(wZ_1-Y_1^{\delta-1}, \tilde{X}_1Y_1+Z_1^{\delta-1}, \tilde{X}_1w+Y_1^{\delta-2}Z_1^{\delta-2})$$

\noindent
is isomorphic to $\frac{1}{\delta(\delta-2)}(1,(\delta-2)(\delta-1)-1)$. Since it is a toric variety, then in particular it is normal so we conclude that it is the normalization of $T_1$, i.e., $T_1^{\nu}$. 
Now $\Z_{\delta}$ acts on $(\tilde{X},Y_1,Z_1,w)$ with weights $(-\rho -1,\rho,1,-\rho -1)$ and we obtain $S_1^{\nu}$ by taking the quotient of $T_1^{\nu}$ under this action. To do this, we use the previous identification $$\tilde{X}=p^{\delta(\delta-2)}, Z_1=p^{\delta-1}q, Y_1=pq^{\delta-1} \text{ and } w=q^{\delta(\delta-2)}$$ 
Note that $\Z_{\delta}$ acts monomially on $T_1^{\nu}$ so the quotient will be again a toric variety and if $L_{T_1^{\nu}}$  is the lattice of $T_1^{\nu}$ , then $L_{T_1^{\nu}}$ is a sublattice of $\Z^2_{p,q}$. Then we can look for invariant monomials in $L_{T_1^{\nu}}$. It is enough to find two subsequent invariant monomials, as it is known that these will generate the lattice for $S_1^{\nu}$. Note that $w^{\delta}=q^{\delta^2 (\delta-2)}$ is invariant. If we find an invariant monomial of the form $w^tY_1=pq^{t\delta(\delta-2)+(\delta-1)}$ then these two points of $L_{T_1^{\nu}}$ would form a basis of the lattice of $S_1^{\nu}$. Now $w^tY_1$ is invariant if and only if $$t(-\rho -1)+\rho \equiv 0 \text{ (mod $\delta$)}.$$ 
If $gcd(\rho+1,\delta)=1$ then the congruence has a unique solution and we conclude that $$S_1^{\nu} \simeq \frac{1}{\delta^2(\delta-2)}(1,\delta^2(\delta-2)-t\delta(\delta-2)-\delta +1).$$
\vspace{0.2cm}

\noindent
If $gcd(\rho+1,\delta)=d>1$, then $L_{T_1^{\nu}}$ is a sublattice of $\Z^2_{p^d,q^d}$. Note that $w^{\delta/d}$ is invariant and so we need to find an invariant monomial of the form $w^tY_1^d=p^dq^{t\delta(\delta-2)+d(\delta-1)}$. Now this monomial is invariant if and only if
$$t(-\rho-1)+d\rho \equiv 0 \text{(mod $\delta$)}$$
Say $\rho+1=dh$ and $\delta=dj$, then the previous congruence is equivalent to 
$$-th+\rho \equiv 0 \text{(mod $\delta$)}$$
which has a unique solution and we conclude that 
 $$S_1^{\nu} \simeq \frac{1}{j\delta(\delta-2)}(1,j\delta(\delta-2)-t\delta(\delta-2)-d(\delta -1)).$$
\vspace{0.2cm}

\noindent
For the other chart, first note that at the origin the tangent cone is given by $$T_{(0,0,0)}=(Y_2Z_2=0).$$
Then at the origin $T_2$ is analytically isomorphic to its tangent cone, thus its quotient by the $\Z_F$ action at the origin will be analytically isomorphic to the quotient  $$(Y_2Z_2=0)\subset \frac{1}{F}(\lambda,1,-1)$$ which is an orbifold normal crossing.


\vspace{0.2cm}

\noindent
Along $C^-=(Y_2=Z_2=0)$, if $X_2\neq 0$, then the $\Z_F$ action is free and therefore at these points the surface has the same singularities as the corresponding points of $S_2^{\nu}$, namely two transversal branches.
\vspace{0.2cm}

\noindent
Away from $C^-$, if $X_2\neq 0$ then the action is free and therefore these are smooth points. Away from $C^-$, if $X_2=0$ then the action is also free so the points of the form $(0,Y_2,0)$ and $(0,0,Z_2)$ are smooth.
\end{proof}

\section{Special fiber of the antiflip with $\delta=2$}
From now on let $\delta=2$. So
$$T_1=(X_1Y_1Z_1=Y_1^2+Z_1^2)\subset \A^3$$ and  $$T_2=(Y_2Z_2=X_2^{m_2'}Y_2^2+X_2^{m_1'}Z_2^2)\subset \A^3.$$ Then $S_1^{\nu}$ and $S_2^{\nu}$ are obtained by taking the quotient of $T_1^{\nu}$ and $T_2^{\nu}$  by the $\Z_2$ and the $\Z_{F}$ actions respectively. 

\begin{proposition}
$S_1$ has two pinch points on $C^-=(Y_1=Z_1=0)$ and it has normal crossings elsewhere along this line. 
\end{proposition}
\begin{proof}
In the first chart, we take the quotient by either $\frac{1}{2}(1,0,1)$ if $\rho$ is even or $\frac{1}{2}(0,1,1)$ if $\rho$ is odd. 
Notice that $T_1$ has pinch points at $(2,0,0)$ and $(-2,0,0)$.
If $\rho$ is even the action interchanges the pinch points and is free everywhere except on the points where $X_1=0$. 
Now $$\frac{1}{2}(1,0,1) \simeq A_1 \times \A^1_{Y_1}$$
with $A_1=(uw=v^2)$ where $u=X_1^2$, $v=X_1Z_1$ and $w=Z_1^2$. Then the equation of $T_1$ can be written as $$Y_1v=Y_1^2+w$$ and then the quotient of $T_1$ by $Z_2$ is given by the spectrum of
$$\C[Y_1,u,v,w]/(uw-v^2, Y_1v-Y_1^2-w)$$ 
Now we can write $w=Y_1v+Y^2$ so the quotient becomes the spectrum of 
$$\C[Y_1,u,v]/(uvY_1-uY_1^2-v^2)$$
Note that we can rewrite the equation$$uvY_1-uY_1^2=v^2$$
as 
$$u(Y_1-\frac{1}{2}v)^2=v^2(1-\frac{1}{4}u)$$
and since $(1-\frac{1}{4}u)$ is a unit close to the origin, we see that the quotient also has a pinch point at the origin and normal crossings elsewhere on the line $C^-=(Y_1=Z_1=0)=(u=v=0)$. 
\vspace{0.2cm}

\noindent
If $\rho$ is odd the action does not interchange the pinch points and does not produce a pinch point at the origin when we quotient by the action. So in this case $S_1$ has two pinch points at the images of $(2,0,0)$ and $(-2,0,0)$ in the quotient $T_1/\frac{1}{2}(0,1,1)$ and it has normal crossings elsewhere on the line $C^-=(Y_1=Z_1=0)$.
\end{proof}

Recall that $X^+$ is an extremal $P$-resolution of $Y$ and it has at most two Wahl singularities $(m_1',a_1')$ and $(m_2',a_2,)$ along $C^+$. If $\delta=2$, then it follows that 
$$2=\delta=cm_1'm_2'-m_1'a_2'-m_2'a_1'$$
where $-c$ is the self intersection of $C^+$ in the minimal resolution of $X^+$. Also recall that $F=m_1'+m_2'$.

\begin{lemma}
\label{cases}
If $\delta=2$ then $F$ is even and we have the following posibilities:
\begin{itemize}
\item $X^+$ is smooth so we have $m_i',a_i'=1$, $i=1,2$ ,and $c=4$. Therefore $X^+$ is the minimal resolution of the cone over the rational normal curve of degree 4. 

\item $X^+$ has one singularity along $C^+$. Then $m_1',a_1'=1$ and $m_2'=2k+1$, $a_2'=2k-1$ for some $k\in \N$ and $c=2$. 

\item $X^+$ has two singularities along $C^+$. Then $c=1$. 
\end{itemize}
\end{lemma}

\begin{proof}  We will assume that $m_2'\geq m_1'$.
\noindent
If $X^+$ is smooth, then $m_i',a_i'=1$, $i=1,2$. Replacing this values in the formula for $\delta$ we get that $c=4$. Therefore $Y$ is the cone over the rational normal curve of degree 4 and $X^+$ is its minimal resolution. In this case $F=2$.
\vspace{0.2cm}

\noindent
If $X^+$ has only one singularity, then $m_1',a_1'=1$ and $m_2'\geq 2$. Replacing in the formula for $\delta$ we get 
$$2=(c-1)m_2'-a_2'$$
Note that $c=1$ is not possible since we would get that $a_2'<0$. Also $c\geq 3$ is not possible since $m_2'-a_2'\geq 1$ then we get that 
$$2=(c-1)m_2'-a_2' \geq 2m_2'-a_2' \geq m_2'+1\geq 3$$ 
Therefore we must have that $c=2$ and then 
$$2=m_2'-a_2'$$
But $gcd(m_2',a_2')=1$, then we must have that they are consecutive odd numbers, i.e., $m_2'=2k+1$ and $a_2'=2k-1$ for some $k\in \N$. In this case $F=2k+2$.
\vspace{0.2cm}

\noindent
Finally, if $X^+$ has two singularities, then $1<m_1'\leq m_2'$. Note that $c\geq 2$ is not possible since we would get 
\begin{eqnarray*}
2 &=& cm_1'm_2'-m_1'a_2'-m_2'a_1'\\
   &\geq & 2m_1'm_2'-m_1'a_2'-m_2'a_1'\\
   &= & m_1'(m_2'-a_2')+m_2'(m_1'-a_1') 
\end{eqnarray*}
but $m_i'-a_i'\geq 1$, $i=1,2$ so we get that 
$$2\geq m_1'+m_2'\geq 4$$
Therefore we must have that $c=1$.
To show that $F$ is even, we will show that $m_1'$ and $m_2'$ are either both even or both odd. Assume otherwise, say $m_1'=2l$ and $m_2'=2j+1$ (the case $m_1'$ odd and $m_2'$ even follows the same argument).
Then $m_1'm_2'$ and $m_1'a_2'$ are even and since $2=m_1'm_2'-m_1'a_2'-m_2'a_1'$, then $m_2'a_1'$ has to be even. But $m_2'$ is odd so we conclude that $a_1'$ has to be even. But this contradicts the fact that $gcd(m_1',a_1')=1$. Therefore $m_1'$ and $m_2'$ are either both even or both odd, and in any of these cases we get that $F$ is even.
\end{proof}

\begin{observation}
\label{k3}
For an initial extremal neighbourhood of type $k2A$, the number $k$ in Definition \ref{seq} is always equal to 3, so $\lambda =c(k-1)-c(k)=c(2)-c(3)$, which we will assume from now on. 
\end{observation}

Now we look at the chart $S_2$. Note that the proof done in the case $\delta \geq 3$ also works in this case. So $S_2^{\nu}$ is given by 

$$(Y_2Z_2=0)\subset \frac{1}{F}(\lambda,1,-1).$$
We will analyze the minimal resolution of $S_2^{\nu}$ in the three cases of Lemma \ref{cases}. 

\begin{proof}[Proof of Proposition \ref{rnc4}]
If $X^+$ is smooth, we have $F=2$ and $\lambda=2$, therefore in this case we have $$(Y_2Z_2=0)\subset \frac{1}{2}(0,1,-1)$$
so $S_2^{\nu}$ is smooth. The fact that $C^2=(-4)$ follows from the fact that $(C^+)^2=(-4)$. 
\end{proof}

\begin{proposition}
\label{onesing}
If $X^+$ has one singularity, then $S_2^{\nu}$ has conjugate cyclic quotient singularities $\frac{1}{k+1}(1,1)$ and $\frac{1}{k+1}(1,-1)$ for some $k\geq 1$, its minimal resolution is equal to the minimal resolution of $Q\in Y$ and $C^2=(-5)$.
\end{proposition}

\begin{proof}
If $X^+$ has one singularity, by Lemma \ref{cases} we have $m_1'=a_1'=1$, $m_2'=2k+1$, $a_2'=2k-1$ for some $k\geq 1$ and $F=2k+2$. One can also check that $\lambda=2$, therefore in this case we have that $S_2^{\nu}$ is given by $$(Y_2Z_2=0)\subset \frac{1}{k+1}(1,1,-1)$$
so it has conjugate singularities $$\frac{1}{k+1}(1,1), \text{ and } \frac{1}{k+1}(1,-1).$$
\vspace{0.2cm}

\noindent
We will prove that the Hirzebruch-Jung continued fraction corresponding to the Wahl singularity with $m_2'=2k+1$ and $a_2'=2k-1$ for $k\geq 1$ is of the form 
 $$[\underbrace{2, \ldots , 2}_{k-1},5,k+1]$$

We will use the following well-known facts about Wahl singularities: If $[a_1, \ldots , a_r]$ is the continued fraction of a  Wahl singularity $(m,a)$, then the conjugate cyclic quotient singularity is the Wahl singularity $(m, m-a)$ and its continued fraction is $[a_r, \ldots , a_1]$.
Then it is enough to prove that the Hirzebruch-Jung continued fraction corresponding to the Wahl singularity with $m=2k+1$ and $a=2$ for $k\geq 1$ is of the form 
$$[k+1, 5, \underbrace{2, \ldots , 2}_{k-1}]$$
If $m=2k+1$ and $a=2$ then note that
$$\frac{(2k+1)^2}{2(2k+1)-1}=\frac{4k^2+4k+1}{4k+1}=(k+1)-\frac{k}{4k+1}$$
so $k+1$ is the first number in the continued fraction. Then
$$\frac{4k+1}{k}=5-\frac{k-1}{k}$$
so $5$ is the second number in the continued fraction and now we are left with $\frac{k}{k-1}$ which is a $A_{k-1}$ singularity and we know that its continued fraction is $[2,\ldots, 2]$ where we have $k-1$ curves. Therefore we conclude that  
$$\frac{(2k+1)^2}{2(2k+1)-1}=[k+1, 5, \underbrace{2, \ldots , 2}_{k-1}]$$
\vspace{0.3cm}

\noindent
Now $X^+$ has a Wahl singularity with $m_2'=2k+1$ and $a_2'=2k-1$ which is represented by the continued fraction $$\frac{m_2'^2}{m_2'a_2'-1}=[k+1, 5, \underbrace{2, \ldots ,2}_{k-1 \text{ curves}}].$$
In Lemma \ref{cases} we showed that in this case $c=2$, therefore in the minimal resolution of $X^+$ we have a configuration of rational curves 
$$[k+1, 5, \underbrace{2, \ldots ,2}_{k \text{ curves}}]$$
Note that 
$$\frac{1}{k+1}(1,1)=[k+1], \text{ and } \frac{1}{k+1}(1,1)=\underbrace{[2, \ldots ,2}_{k \text{ curves}}]$$
so the minimal resolution of $X^+$ is equal to the minimal resolution of $S_2^{\nu}$ and we conclude then the proper transform of $C^-$ has to be the $(-5)$-curve in this configuration.  
\end{proof}

\begin{proposition}
\label{twosing}
If $X^+$ has two singularities, then $S_2^{\nu}$ has conjugate cyclic quotient singularities $\frac{1}{f}(p,1)$ and $\frac{1}{f}(p,-1)$ where $f=F/2$ and $p=\lambda/2$. Its minimal resolution is equal to the minimal resolution of $Q\in Y$ and $C^2=(-5)$.
\end{proposition}

\begin{proof}
If $X^+$ has two singularities, then $F$ is even and $c(2)-c(3)=m_2'-a_2'+a_1'$ is also even. Therefore in this case we have 
$$(Y_2Z_2=0)\subset \frac{1}{f}(p,1,-1)$$
where $f=F/2$ and $p=\lambda/2=(m_2'-a_2'+a_1')/2$, so $S_2^{\nu}$ has conjugate singularities $$\frac{1}{f}(p,1), \text{ and } \frac{1}{f}(p,-1).$$
Now we need to prove that in the minimal resolution of $(X^-)^{\nu}$, the proper transform of $C^-$ is a $(-5)$-curve.
\vspace{0.2cm}

\noindent
Let $\tilde{C}$ be the proper transform of $C^-$ in $X^{\nu}$ and $C$ be the proper transform of $\tilde{C}$ in $\bar{X}$. In $\bar{X}$, let $A$ and $B$ be the exceptional curves that intersect $C$ obtained by resolving the singularities of $X^{\nu}$. Let $\pi \colon \bar{X} \to X^{\nu}$, then 
$$\pi^* \tilde{C}=C+\frac{p}{f}A+\frac{f-p}{f}B$$
so using the projection formula we have
$$\tilde{C}^2=(\pi^* \tilde{C})(\pi_* C)= C^2+\frac{p}{f}+\frac{f-p}{f}=C^2+1$$
Therefore, proving that $C^2=-5$ is equivalent to proving that $\tilde{C}^2=-4$. 
\vspace{-0.2cm}

\noindent
Now germs $(C^- \subset X^-)$ and $(C^+\subset X^+)$ have $\Q$-Gorenstein smoothings with the same general fiber, then we must have $(K_X^+)^2=(\omega_{X^-})^2$, see \cite{L86} for a definition of the self-intersection in the non-compact case. 
Let $\pi^+ \colon X^+ \to Y$ and $\pi^- \colon X^- \to Y$. Then 
$$K_{X^+}=(\pi^+)^* K_Y+aC^+$$
and 
$$K_{X^{\nu}}=(\pi^-)^* K_Y +b\tilde{C}.$$
We abuse notation and let $\pi^- \colon X^{\nu} \to Y$. Then 
$$\nu^* \omega_{X^-}=K_{X^{\nu}}+\tilde{C}=(\pi^-)^* K_Y +(b+1)\tilde{C}.$$
And since the self-intersections are equal then we must have 
$$a^2(C^+)^2=(b+1)^2(\tilde{C})^2$$
Now 
$$(C^+)^2=-1+\frac{m_1'a_1'-1}{m_1'^2}+\frac{m_2'a_2'-1}{m_2'^2}=-(\frac{1}{m_1'}+\frac{1}{m_2'})^2$$

On the other hand, using the \emph{different formula} (see Section 16 of \cite{K92}), we have 
$$(K_{X^+}+C^+)C^+=-2+(1-\frac{1}{m_1'^2})+(1-\frac{1}{m_2'^2})=-(\frac{1}{m_1'^2}+\frac{1}{m_2'^2})$$
but also
$$(K_{X^+}+C^+)C^+=(a+1)(C^+)^2$$
so from this two expressions we get that 
$$a=\frac{-2m_1'm_2'}{(m_1'+m_2')^2}$$
and we conclude that 
$$a^2(C^+)^2=-\frac{4}{(m_1'+m_2')^2}=-\frac{4}{F^2}$$
We can do the same calculation for $X^{\nu}$ and we get 
$$(K_{X^{\nu}}+\tilde{C})\tilde{C}=-2+2(1-\frac{2}{F})=-\frac{4}{F}$$
but also 
$$(K_{X^{\nu}}+\tilde{C})\tilde{C}=(b+1)\tilde{C}^2$$
so we get that 
$$(b+1)=-\frac{4}{F(\tilde{C})^2}$$
and then 
$$(b+1)^2(\tilde{C})^2=-\frac{16}{F^2(\tilde{C})^2}.$$
Finally 
$$-\frac{4}{F^2}=a^2(C^+)^2=(b+1)^2(\tilde{C})^2=-\frac{16}{F^2(\tilde{C})^2}$$
and we get that $\tilde{C}^2=-4$ and then $C^2=-5$.
\end{proof}

\begin{observation}
In the previous propositions we have assumed that $m_2'\geq m_1'$, but the same proofs work if $m_1' \geq m_2'$. In fact, as it will be shown in the next proposition, when $\delta=2$, then $Q\in Y$ has two extremal P-resolutions, one with $m_2'> m_1'$ and one with $m_1'> m_2'$ unless $Q\in Y$ is the cone over the rational normal curve of degree 4 which has only one extremal P-resolution. 
\end{observation}


\begin{proposition}
\label{fp}
Given integers $f,p$ with $1\leq p \leq \frac{f}{2}$ and $gcd(p,f)=1$, there are two extremal $P$-resolutions $X^+$ with Wahl singularities $(m_1',a_1')$ and $(m_2', a_2')$ such that:
\begin{enumerate}
\item $\delta=2$
\item If $\X^+$ is a $\Q$-Gorenstein smoothing of $X^+$ with axial multiplicities $\alpha_1=\alpha_2$ and $\X^-$ is the antiflip, then the normalization of the special fiber $X^-$ has singularities
$$\frac{1}{f}(p,1), \text{ and } \frac{1}{f}(p,-1).$$
\end{enumerate}
\end{proposition}
\begin{proof}

If $p=1$, then the two extremal P-resolutions have singularities given by the data
\begin{enumerate}
\item[(i)] $m_1'=a_1'=1$, $m_2'=2f-1$ and $a_2'=2f-3$.
\item[(ii)] $m_1'=f+1$, $a_1'=1$, $m_2'=f-1$ and $a_2'=f-2$.
\end{enumerate}
In the first case $X^+$ has only one singularity and the second case $X^+$ has two singularities. For both surfaces we have that $\delta=2$ and by Proposition \ref{onesing} and Proposition \ref{twosing}  it follows that $(X^-)^{\nu}$ has singularities 
$$\frac{1}{f}(p,1), \text{ and } \frac{1}{f}(p,-1)$$. 
\vspace{0.2cm}

\noindent
If $1<p<f/2$, since $gcd(p,f)=1$, then there exist unique integers $m_1',a_1'$ with $1<m_1'<f$ such that 
$$m_1'p-a_1'f=1.$$
Note that $a_1'>0$, otherwise we would get that $m_1'p-a_1'f>1$. Also we must have that $a_1'<m_1'$, otherwise we would get that $f<p$. Define $m_2'=2f-m_1'$ and $a_2'=m_2'+a_1'-2p$. Then we need to check that $gcd(m_2',a_2')=1$ and that $\delta=2$. 
Note that $1<m_1'< m_2'$, therefore we are in the case when $X^+$ has two singularities. Then $c=1$ and 
\begin{eqnarray*}
\delta &=& m_1'm_2'-m_1'a_2'-m_2'a_1'\\
   &=& m_1'm_2'-m_1'(m_2'+a_1'-2p)-(2f-m_1')a_1'\\
   &=& 2(pm_1'-fa_1')\\
   &=& 2 
\end{eqnarray*}

Note that we can write $f$ and $p$ in terms of $m_1',a_1',m_2',a_2'$ as 
\[
\begin{bmatrix}
2f\\
2p
\end{bmatrix}=
\begin{bmatrix}
m_1' & m_2'\\
a_1' & m_2'-a_2'
\end{bmatrix} 
\begin{bmatrix}
1\\
1
\end{bmatrix}
\]  
and note that $$det 
\begin{bmatrix}
m_1' & m_2'\\
a_1' & m_2'-a_2'
\end{bmatrix}
=\delta =2
$$ 
Then 
 $$det 
\begin{bmatrix}
2f & m_2'\\
2p & m_2'-a_2'
\end{bmatrix}
=2
\Rightarrow 
det 
\begin{bmatrix}
f & m_2'\\
cp& m_2'-a_2'
\end{bmatrix}
=1
$$ 
so we conclude that $1=gcd(m_2',m_2'-a_2')=gcd(m_2',a_2')$.
\vspace{0.2cm}

\noindent
For the other extremal P-resolution, using again the fact that $gcd(f,p)=1$, there are unique integers $m_2', q$ with $1<m_2'<f$ such that 
$$qf-m_2'p=1.$$
Clearly $q>0$, and note that $q<m_2'$, otherwise we would get that 
$$1=qf-m_2'p \geq m_2'(f-p)$$
which is not possible since the expression on the right hand side is greater that one. Define $a_2'=m_2'-q$, $m_1'=2f-m_2'$ and $a_1'=2p-q=2p-m_2'+a_2'$. Then we need to check that $\delta=2$ and $gcd(m_1',a_1,)=1$. Note that $gcd(m_2',a_2')=1$ and $1<m_2'<m_1'$, therefore we are in the case when $X^+$ has two singularities. Then $c=1$ and 
\begin{eqnarray*}
\delta &=& m_1'm_2'-m_1'a_2'-m_2'a_1'\\
   &=& m_1'm_2'-m_1'(m_2'-q)-m_2'(2p-q)\\
   &=& q(m_1'+m_2')-m_2'p\\
   &=& 2(qf-m_2'p)\\
   &=& 2 
\end{eqnarray*}

\noindent
Note that we can write $f$ and $p$ in terms of $m_1',a_1',m_2',a_2'$ as 
\[
\begin{bmatrix}
2f\\
2p
\end{bmatrix}=
\begin{bmatrix}
m_1' & m_2'\\
a_1' & m_2'-a_2'
\end{bmatrix} 
\begin{bmatrix}
1\\
1
\end{bmatrix}
\]  
and note that $$det 
\begin{bmatrix}
m_1' & m_2'\\
a_1' & m_2'-a_2'
\end{bmatrix}
=\delta =2
$$ 
Then 
 $$det 
\begin{bmatrix}
m_1' & 2f\\
a_1' & 2p
\end{bmatrix}
=2
\Rightarrow 
det 
\begin{bmatrix}
m_1' & f\\
a_1' & p
\end{bmatrix}
=1
$$ 
so we conclude that $1=gcd(m_1',a_1')$.
\end{proof}

\begin{remark}
For a given pair $(f,p)$, if the minimal resolution of the normalization of $X^-$ has a chain of exceptional divisors $[q_1,\ldots, q_l]$, then note that for the pair $(f,f-p)$, the minimal resolution of the normalization of $X^-$ has the same chain of exceptional divisors but written in the opposite order, i.e. $[q_l,\ldots ,q_1]$. These two singularities are isomorphic and therefore, the corresponding special fibers $X^-$ are also isomorphic. For example, if $p=f-1$, then the two extremal P-resolutions have singularities given by the data
\begin{enumerate}
\item[(i)] $m_1'=f-1$, $a_1'=f-2$, $m_2'=f+1$ and $a_2'=1$.
\item[(ii)] $m_1'=2f-1$, $a_1'=2f-3$, and $m_2'=a_2'=1$.
\end{enumerate}
In the first case $X^+$ has two singularities and the second case $X^+$ has one singularity. For both surfaces we have that $\delta=2$ and by Proposition \ref{twosing} and Proposition \ref{onesing}  it follows that $(X^-)^{\nu}$ has singularities 
$$\frac{1}{f}(p,1), \text{ and } \frac{1}{f}(p,-1).$$
\end{remark}

\begin{proof}[Proof of Theorem \ref{delta2}]
Using Propositions \ref{onesing} and \ref{twosing}, we show how to associate a pair $(f,p)$ to a $\Q$-Gorenstein smoothing with equal axial multiplicities of a given extremal P-resolution $X^+$. Using Proposition \ref{fp}, we see that to each pair $(f,p)$ we can associate two different extremal P-resolutions, thus obtaining the two-to-one correspondence.
\end{proof}

\begin{remark}
Proposition \ref{fp} shows that, with the exception of the cone over the rational normal curve of degree 4, every cyclic quotient singularity $Q\in Y$ having an extremal P-resolution with $\delta=2$, in fact has exactly two extremal P-resolutions (recall that a cyclic quotient singularity has at most two extremal P-resolutions \cite{HTU}). This fact, together with Lemma \ref{cases} show that all these singularities satisfy the conditions of Theorem 1.2 and Theorem 1.3 of \cite{UV}, so in particular they satisfy the Wormhole conjecture (see Conjecture 1.1 in \cite{UV}).
\end{remark}


\vspace{0.2cm}

\small Department of Mathematics and Statistics, University of Massachusetts, Amherst, MA, USA.

\end{document}